\documentclass[a4paper,12pt]{article}

\usepackage{amsthm,amsmath,stmaryrd,bbm,geometry}
\usepackage[utf8]{inputenc}
\usepackage{amssymb}
\usepackage[english]{babel}
\usepackage{graphicx}
\usepackage{amsfonts,amssymb}
\usepackage{verbatim}
\usepackage{enumitem}

\newcommand{\po}{\left(}
\newcommand{\pf}{\right)}
\newcommand{\co}{\left[}
\newcommand{\cf}{\right]}
\newcommand{\cco}{\llbracket}
\newcommand{\ccf}{\rrbracket}
\newcommand{\R}{\mathbb R}

\newcommand{\N}{\mathbb N} 
\newcommand{\dd}{\text{d}}
\newcommand{\na}{\nabla}

\newtheorem{thm}{Theorem}

\newtheorem{rem}{Remark}

\newtheorem{prop}[thm]{Proposition}

\title{A note on a Vlasov-Fokker-Planck equation with non-symmetric interaction}
\author{Pierre Monmarché}

\begin{document}
\maketitle

\begin{abstract}
In the recent \cite{LudoMaxime}, Cesbron and Herda study a Vlasov-Fokker-Planck (VFP) equation with non-symmetric interaction, introduced in physics to model the distribution of electrons in a synchrotron particle accelerator. We make four remarks in view of their work: first, it is noticed in \cite{LudoMaxime} that  the free energy classically considered for the (symmetric) VFP equation is \emph{not} a Lyapunov function in the non-symmetric case, and we will show however that this is still the case for a suitable definition of the free energy (with no explicit expression in general). Second, when the interaction is sufficiently small (in $W^{1,\infty}$), it is proven in \cite{LudoMaxime} that the equation has a unique stationary solution which is locally attractive; in this spirit,  we will see that, when the interaction force is Lipschitz with a sufficiently small constant, the convergence is  global. Third, we also briefly discuss the mean-field interacting particle system corresponding to the VFP equation which, interestingly, is a non-equilibrium Langevin process. Finally, we will see that, in the small interaction regime, a suitable (explicit)  non-linear Fisher information is contracted at constant rate, similarly to the situation of Wasserstein gradient flows for convex functionals (although here the dynamics is not a gradient flow).
\end{abstract}

\section{Introduction}

The Vlasov-Fokker-Planck (VFP) equation studied in \cite{LudoMaxime} reads
\begin{equation}\label{eq:VFP}
\left\{\begin{array}{l}
\partial_t f + v\cdot \partial_x f + \po F_{f}(t,x) -   x\pf  \partial_v f =  \gamma \partial_v \cdot \po v f + \partial_v f\pf \\
\\
F_f(t,x) = -\lambda \int_{\R^2} \partial_x K(x-y) f(t,y,w) \dd y \dd w\\
\\
f(0,x,v) = f^{\text{in}}(x,v)
\end{array}\right.
\end{equation}
where $f(t,x,v)$ stands for the density of electrons at time $t\geqslant 0$, position $x\in\R$ and velocity $v\in\R$, $\gamma,\lambda$ are positive parameters and $K\in \mathcal C^2(\R,\R)$ is an interacting potential, with $\|\partial_x^2 K\|_\infty <\infty$. We assume that the initial density $f_{\text{in}}$ has finite second moment. In fact, in \cite{LudoMaxime}, the equation involves additional physical parameters $\theta,\alpha$ to describe the temperature  and external potential intensity of the system, but we normalized them to $1$, which amounts to rescale time and position. Similarly,  in the present work, by contrast to \cite{LudoMaxime}, we assume that the mass of the initial condition is normalized to $\int_{\R^2} f^{\text{in}} =1$ (which amounts to a rescaling of $\lambda$, since the mass is preserved by the equation). We refer to  \cite{LudoMaxime} for physical motivations, well-posedness results and further references. This equation has been extensively studied in various context but, classically, motivated by the reciprocity principle of forces in classical physics, the potential $K$ is assumed to be even, while  \cite{LudoMaxime} focuses on the case where it is not (motivated by models of electrons in a synchrotron particle accelerator, cf. \cite{LudoMaxime}). From a mathematical perspective, this leads to interesting consequences.  Classically (see e.g. \cite{GuillinWuZhang,CMCV,LudoMaxime}), the free energy associated to \eqref{eq:VFP} is defined as
\begin{multline*}
\mathcal E(f) = \int_{\R^2} f(x,v) \ln f(x,v) \dd x \dd v  + \int_{\R^2} \frac{v^2 + x^2}{2} f(x,v)\dd x \dd v \\
+ \lambda \int_{\R^4} K(x-y) f(x,v) f(y,w)\dd x\dd v \dd y \dd w  \,,
\end{multline*}
and, when $K$ is even, it is known to decay along the flow of \eqref{eq:VFP} (i.e. the free energy is  a Lyapunov function for \eqref{eq:VFP}). However, \cite[Corollary 3.6]{LudoMaxime} shows that, in fact, as soon as $K$ is not even then, whatever the intensity $\lambda$ of the interaction, there exist initial conditions such that $t\mapsto \mathcal E(f(t,\cdot))$ initially increases. Notice that $\mathcal E(f)$ is unchanged if we replace the potential $K$ by the symmetrized $x\mapsto (K(x)+K(-x))/2$. This problem is reminiscent of other McKean-Vlasov models where the construction of suitable non-linear free energy functions is unclear, see e.g. \cite{YulongLu}, which gives a general motivation for the present note. Following Large Deviation theory (e.g. \cite{GuillinWuZhang}), the free energy should be obtained as the limit as $N\rightarrow \infty$ of the scaled relative entropy with respect to equilibrium for the system of $N$ interacting particles corresponding to~\eqref{eq:VFP}. The latter is the Markov process $(X,V)=(X^1,\dots, X^N,V^1,\dots V^N)\in\R^{2N}$ solving
\begin{equation}\label{eq:mean_field_particle}
\forall i\in\cco 1,N\ccf\,,\quad \left\{
\begin{array}{rcl}
\dd X^i_t &=& V^i_t\dd t\\
\dd V^i_t &=  & -\po X_t^i + \gamma V_t^i + \frac{\lambda}{N-1} \sum_{j\neq i} \partial_x K(X^i_t-X^j_t) \pf \dd t + \sqrt{2} \dd B_t^i\,, 
\end{array}
\right.
\end{equation}
where $B=(B^1,\dots,B^N)$ is a standard $N$-dimensional Brownian motion. Equivalently, this reads
\begin{equation}\label{eq:mean_field_particle_2}
 \left\{
\begin{array}{rcl}
\dd X_t &=& V_t\dd t\\
\dd V_t &=  & -\po X_t + \gamma V_t + \lambda F(X_t)\pf \dd t + \sqrt{2} \dd B_t\,, 
\end{array}
\right.
\end{equation}
where, for all $i\in\cco 1,N\ccf$,
\[F_i(x_1,\dots,x_N) =  \frac{1}{N-1} \sum_{j \in\cco 1,n\ccf\setminus\{i\}} \partial_x K(x_i-x_j)\,. \]
We write $(P_t^N)_{t\geqslant 0}$ the associated Markov semi-group, so that $\nu P_t^N$ is the law of $(X_t,V_t)$ when $(X_0,V_0)$ is distributed according to $\nu$.

When $K$ is even, $F= \na U_N$ where
\[U_N(x_1,\dots,x_n) = \frac{1}{2(N-1)} \sum_{i,j\in\cco 1,N\ccf} K(x_i-x_j)\,,\]
in which case \eqref{eq:mean_field_particle_2} is a classical Langevin diffusion on $\R^{2N}$ with invariant density proportional to $\exp (-|x|^2/2 - U_N(x)-|v|^2/2)$ (under suitable conditions on $K$),  but in general if $K$ is not even then $F$ may not be a gradient, in which case the invariant measure of \eqref{eq:mean_field_particle_2} doesn't have an explicit form  (see the discussion in \cite{M35}). In that case, \eqref{eq:mean_field_particle_2} is  referred to as a \emph{non-equilibrium} Langevin process \cite{Iacobucci}. Notice that in the non-linear equation \eqref{eq:VFP} however the force is always the gradient of $U_f(x) = \int_{\R^2} K(x-y) f(t,y,w)\dd y\dd w$, in particular the stationary solutions of~\eqref{eq:VFP} are solutions of an explicit fixed-point problem, see \cite[Theorem 1.2]{LudoMaxime}.

The long-time convergence of \eqref{eq:VFP} and \eqref{eq:mean_field_particle_2} (for $\lambda$ small enough) is studied in \cite{BolleyGuillinMalrieu}. More precisely,  since it is based on coupling methods, which do not rely on explicit densities, gradient structures or explicit dual operators, \cite{BolleyGuillinMalrieu} considers a general case where the forces are not necessarily of gradient forms. However, it does assume that the interaction force is odd (which, in our settings, corresponds to the case where $K$ is even), but this has little impact in the analysis (see e.g. \cite{Contraction}). We briefly present below how the coupling argument of \cite{BolleyGuillinMalrieu} still works in the non-symmetric case (in fact, this is also done in the recent \cite[Corollary 4.2]{QianRenWang}). This gives a convergence to equilibrium in Wasserstein distance starting from any distribution with a finite second moment, provided $\|\partial_x^2 K\|_\infty$ is small enough, to look at in relation with the result of \cite{LudoMaxime} which requires the solution to be close to stationarity (assuming however that $\|\partial_x K\|_\infty$ is small enough instead of $\|\partial_x^2 K\|_\infty$, so that the results are not directly comparable). This global weak convergence can then be improved  by regularization thanks to \cite{QianRenWang} to a convergence in terms of the relative entropy,  hence to a strong $L^1$ convergence.
 As a conclusion, writing respectively $\|\cdot\|_{TV}$, $\mathcal W_2$, $\mathcal H$ the total variation norm, $L^2$ Wasserstein distance and relative entropy, based on known arguments, the following holds.

\begin{prop}\label{prop:cv}
Assume that
\begin{equation}
\label{eq:lambda_condition}
\lambda \| \partial_x^2 K\|_\infty \leqslant \frac{\min(\gamma,\gamma^{-1})}{8}\,.
\end{equation}
Then there exist $C,a>0$ such that the following hold.
\begin{enumerate}
\item For all $N\geqslant 2$, the process \eqref{eq:mean_field_particle_2} admits a unique invariant measure $\mu_N$  and, for all initial distributions $\nu$ on $\R^{2N}$ and all $t\geqslant 1$,
\begin{equation}
\label{eq:main_particle_tempslong}
\|\nu P_t^N - \mu_N\|_{L^1(\R^{2N})}^2 +  \mathcal H\po \nu P_t^N | \mu_N\pf + \mathcal W_2^2(\nu P_t^N,\mu_N) \leqslant C e^{-at} \mathcal W_2^2 (\nu,\mu_N)\,.
\end{equation}
\item There exists a unique stationary solution $f_*$ of \eqref{eq:VFP} and, for all solution $f$ of \eqref{eq:VFP} and all $t\geqslant 1$,
\begin{equation}
\label{eq:NL_tempslong}
\|f(t,\cdot) - f_*\|_{L^1(\R^2)}^2 + \mathcal H\po f(t,\cdot) | f_*\pf + \mathcal W_2^2(f(t,\cdot),f_*) \leqslant C e^{-at} \mathcal W_2^2 (f(0,\cdot),f_*)\,.
\end{equation}
\end{enumerate}
\end{prop}

\begin{rem}
In view of the deterministic contraction established along  the proof of Proposition~\ref{prop:cv},  from \cite[Proposition 3]{Contraction}, under \eqref{eq:lambda_condition}, we also get that $\mu_N$ satisfies a so-called log-Sobolev inequality with a constant uniform in $N$. See \cite{Delgadino} on  that topic.
\end{rem}

In the remaining we assume that \eqref{eq:mean_field_particle_2} admits an invariant measure  $\mu_N$  (which is thus the case under the condition \eqref{eq:lambda_condition} or, for instance, using classical Lyapunov arguments, if $\partial_x K$ is bounded or if $K$ is convex outside some compact interval), and we focus on a free energy defined   for any probability measure $\nu$ on $\R^2$ as
\[\mathcal F(\nu) = \liminf_{N\rightarrow \infty}\frac1N \mathcal H\po\nu^{\otimes N}|\mu_N\pf\,.\]
In the non-interacting case were $K=0$, $\mu_N = \mu_1^{\otimes N}$ and we simply get $\frac1N \mathcal H(\nu^{\otimes N}|\mu_N) = \mathcal H(\nu|\mu_1)=\mathcal F(\nu)$. More generally, according to \cite{GuillinWuZhang}, when $K$ is even, $\mathcal F(\nu) = \mathcal E(\nu) + C$ for some constant $C $ (the free energy is in fact always defined up to an additive constant). This is not the case when $K$ is not even and, by contrast to \cite[Corollary 3.6]{LudoMaxime},  $\mathcal F$ is a Lyapunov functional for \eqref{eq:VFP}:

\begin{prop}\label{prop:Lyapunov}
For any $f$ solution to \eqref{eq:VFP}, $t\mapsto \mathcal F(f(t,\cdot))$ is non-increasing. 
\end{prop}

 It is interesting to consider the very simple case where $K(x) =ax^2 + bx $ for some $a,b\in\R$, and $\lambda=1$.  Indeed, in that case,   \eqref{eq:mean_field_particle_2} is an equilibrium Langevin diffusions with equilibrium density proportional to $\exp (-|\mathbf{x}|^2/2 - W_N(\mathbf{x})-|\mathbf{v}|^2/2)$ with 
\[W_N(x_1,\dots,x_N) = \frac{a}{2(N-1)}\sum_{i, j\in\cco 1,N\ccf}(x_i-x_j)^2 + b \sum_{i=1}^N x_i\,. \] 
In fact, the corresponding non-linear equation \eqref{eq:VFP} can be as a Vlasov-Fokker-Planck equation with even interaction potential $a x^2$ and with the external potential $x^2/2 + bx$ (instead of simply $x^2/2$ as in \eqref{eq:VFP}). From \cite{GuillinWuZhang} we get that, in that case,
\begin{multline}
\mathcal F(f) = \int_{\R^2} f(x,v) \ln f(x,v) \dd x \dd v  + \int_{\R^2} \po \frac{v^2 + x^2}{2} + bx\pf f(x,v)\dd x \dd v   \\
+ a \int_{\R^4} (x-y)^2 f(x,v) f(y,w)\dd x\dd v \dd y \dd w +C \,,\label{eq:nonlinF}
\end{multline}
for some constant $C\in\R$. In particular, when $b\neq 0$, we do not recover $\mathcal E$. When $a=0$, the particles \eqref{eq:mean_field_particle} are independent, so that  $\mathcal F(f) = \mathcal H(\nu|\mu_1)$. However, due to the non-linear term in \eqref{eq:nonlinF}, this is not the case when $a\neq 0$. As a conclusion, when both $a$ and $b$ are non-zero, the free energy is neither the usual free energy $\mathcal E$ nor the relative entropy with respect to the stationary solution $\bar{\mu}$, but something else. In the general non-Gaussian non-symmetric case, a priori we do not have an explicit formula for $\mathcal F$. 


\bigskip

Finally, if, instead of  the free energy, we focus on the non-linear Fisher information which classically appears as the free energy dissipation in the symmetric interaction case (at least in the elliptic case, i.e. the granular media equation, which is the overdamped limit of \eqref{eq:VFP} as $\gamma\rightarrow \infty$ when accelerating time by a factor $\gamma$ and averaging out the velocities), we see that we can obtain a contraction result in the spirit of  \cite{CMCV} for an explicit quantity of the same nature. More precisely, for a $2\times 2$ matrix $A$ and a probability density $f$, let
\[ \mathcal I_A(f) = \int_{\R^{2}} \left|A \na \ln \frac{f}{\hat f}\right|^2 f \]
where
\begin{equation}
\label{eq:fhat}
\hat f(x,y)=  \frac{e^{-\frac{1}{2}|x|^2 - \lambda K\star f(x) -\frac{1}{2}|v|^2  }}{Z_f}\,,\qquad Z_f = \int_{\R^{2}} e^{-\frac{1}{2}|x|^2 - \lambda K\star f(x) -\frac{1}{2}|v|^2  }\dd x \dd v
\end{equation}
stands for the invariant measure of the Markov generator
\[L_f = v \partial_x - (x+\partial_x K\star f(x)+\gamma v)\partial_v + \gamma \partial_v^2\]
(here we write $K\star f(x,v) = \int_{\R^2} K(x-y) f(y,w)\dd y \dd w$).
\begin{prop}\label{prop:Fisher}
There exist $c>0$ and an invertible $2\times 2$ matrix $A$ (depending only on $\gamma$) such that, under the condition \eqref{eq:lambda_condition}, for any solution $f$ of \eqref{eq:VFP} and any $t\geqslant \geqslant 0$,
\[  \mathcal I_A\po f(t,\cdot)\pf \leqslant e^{- c(t-s)}  \mathcal I_A\po f(s,\cdot)\pf\,. \]
\end{prop}
Using the specific form of $A,c$ given in the proof of this result, we can get for the full non-linear Fisher information (i.e. the case $A=I$)  the explicit exponential decay
\[  \mathcal I_{I}\po f(t,\cdot)\pf  \leqslant 4 \exp \po - \min(\gamma,\gamma^{-1}) \frac{t}8 \pf  \mathcal I_{I}\po f(0,\cdot)\pf\,. \] 

\section{Proofs}

\begin{proof}[Proof of Proposition~\ref{prop:cv}]
We start by the analysis of the particle system \eqref{eq:mean_field_particle_2}. Denoting by $\na F$ the Jacobian matrix of $F$ and $|\na F|$ its operator norm, we see that, for all $x,u\in\R^{n}$,
\begin{eqnarray*}
|\na F(x) u|^2 & =& \frac{1}{(N-1)^2}\sum_{i=1}^n \left|\sum_{j\neq i} \partial_x^2 K(x_i-x_j)(u_i-u_j) \right |^2 \\
& \leqslant  & \frac{\|\partial_x^2 K\|_\infty^2}{N-1}\sum_{i=1}^n \sum_{j\neq i} |u_i-u_j|^2 \ \leqslant  \  4\|\partial_x^2 K\|_\infty^2 \sum_{i=1}^n  |u_i|^2\,,
\end{eqnarray*}
namely $|\na F(x)|\leqslant 2\|\partial_x^2 K\|_\infty$. Let $(Z_t)_{t\geqslant 0}= (X_t,V_t)_{t\geqslant 0}$ and  $(\tilde Z_t)_{t\geqslant 0}= (\tilde X_t,\tilde V_t)_{t\geqslant 0}$ be two solutions of \eqref{eq:mean_field_particle_2} with different initial conditions but driven by the same Brownian motion. For a parameter $a\in\R$ to be chosen later on, we set 
\[P_t = X_t - \tilde X_t + a \po V_t - \tilde V_t\pf\,,\qquad Q_t = V_t-\tilde V_t\,.\]
Straightforward computations give, for any $b>0$,
\begin{multline}
\frac12 \partial_t \po |P_t|^2 + b|Q_t|^2 \pf
= -a |P_t|^2 + b(a-\gamma)|Q_t|^2 + (1+a^2 - a\gamma-b) P_t\cdot Q_t \\
- \lambda(aP+bQ) \cdot \po F(X_t) - F(\tilde X_t)\pf \,.\label{eq:mult}
\end{multline}
Bounding $|F(X_t) - F(\tilde X_t)| \leqslant 2\|\partial_x^2 K\|_\infty |P-aQ| $, we take for simplicity  $a= \min(\gamma,\gamma^{-1})/2 $ and $b= 1+a^2 - a\gamma$ to get, writing $\lambda' = 2 \|\partial_x^2 K\|_\infty \lambda$, 
\begin{eqnarray*}
\frac12 \partial_t  \po |P_t|^2 + b|Q_t|^2 \pf &\leqslant& -a |P_t|^2 - \frac{b\gamma}2 |Q_t|^2 + \lambda' |aP+bQ||P-aQ|  \\
& \leqslant & -a(1-\lambda') |P_t|^2 - \po \frac{b\gamma}2 - \lambda' ab \pf  |Q_t|^2 + \lambda'(a^2+b)|P_t||Q_t| \\ 
& \leqslant & -\frac{3}{4}a  |P_t|^2 - \frac{3b\gamma}8  |Q_t|^2 +  \frac34\lambda' \po  |P_t|^2 + |Q_t|^2\pf\,, 
\end{eqnarray*}
where we used that $a=\min(\gamma,\gamma^{-1})/2 \leqslant 1/2$, that $b+a^2 \leqslant 3/2$ and that $\lambda'\leqslant 1/4$. Using that $b\in[ 1/2,5/4]$ and $\lambda' \leqslant a/4 \leqslant b \gamma/4$ gives
\[ \frac12 \partial_t   \po |P_t|^2 + b|Q_t|^2 \pf \leqslant -\frac{a}{2}   |P_t|^2 - \frac{b\gamma}{16}  |Q_t|^2 \leqslant - \frac{a}{8}\po |P_t|^2 + b|Q_t|^2\pf\,. \] 
Going back to the Euclidean coordinates  using the equivalence between norms leads to
 \[ |Z_t - \tilde Z_t|^2 \leqslant \frac{ \max(2,b+2a^2)}{\min \po \frac12,\frac{b}{1+2 a^2}\pf} e^{-\frac{a}{4} t } |Z_0 - \tilde Z_0|^2 = 4  e^{-\frac{a}{4} t } |Z_0 - \tilde Z_0|^2\,. \] 
This classically implies
\begin{equation}
\label{eq:W2contract}
\mathcal W_2(\nu P_t^N, \tilde \nu P_t^N) \leqslant 2  e^{-\frac{a}{8} t } \mathcal W_2(\nu P_t, \tilde \nu P_t)
\end{equation}
for all initial distributions $\nu,\tilde\nu$ on $\R^{2N}$, and then the existence of a unique invariant measure $\mu_N$ by the Banach fixed-point theorem and a semi-group argument. Moreover, thanks to \cite[Corollary 4.7]{GuillinWang}, there exists a constant $C$ such that   
\[\mathcal H\po \nu P_t^N | \mu_N\pf \leqslant C \mathcal W_2^2 (\nu,\mu_N)\,,\]
for all distribution $\nu$  on $\R^{2N}$ and all $t\geqslant 1$ (Moreover the constant $C$ does not depend on $N$ here, as mentionned after \cite[Proposition 15]{GuillinMonmarche}). Combining this with \eqref{eq:W2contract} and Pinsker's inequality concludes the proof of  \eqref{eq:main_particle_tempslong}.

Using a synchronous coupling as in \cite[Proposition 11]{GuillinMonmarche} (which doesn't rely on the fact the force is a gradient and thus applies readily in the present case) we get that, for any initial density $f\in\mathcal P_2(\R^{2})$ (the set of probability measures with finite second moment) and any $t>0$, there exists $C_t >0$ such that
\[\mathcal W_2\po f^{\otimes N}(t,\cdot), f^{\otimes N}(0,\cdot) P_t^N \pf \leqslant C_t\,.\]
Dividing \eqref{eq:W2contract} by $N$ (with $\nu=f^{\otimes N}(0,\cdot)$, $\tilde \nu= \tilde f^{\otimes N}(0,\cdot)$), using that $\mathcal W_2^2 (f^{\otimes N},\tilde f^{\otimes N}) = N \mathcal W_2^2(f,\tilde f)$, sending $N\rightarrow \infty$ gives
\[\mathcal W_2(f(t,\cdot), \tilde f(t,\cdot)  ) \leqslant 2  e^{-\frac{a}{8} t } \mathcal W_2(f(0,\cdot), \tilde f(0,\cdot))\]
for any solutions $f,\tilde f$ of \eqref{eq:VFP} with initial condition in $\mathcal P_2(\R^{2})$. Again, this implies the existence of a unique stationary solution for \eqref{eq:VFP} by a fixed-point argument, and the proof of \eqref{eq:main_particle_tempslong} is concluded by applying the Wasserstein/entropy regularization for \eqref{eq:VFP} stated in  \cite[Theorem 4.1]{QianRenWang}. 
\end{proof}
 
 \begin{proof}[Proof of Proposition~\ref{prop:Lyapunov}]
Fix a solution $f$ to \eqref{eq:VFP}. For $N\in\N$, $t\geqslant0$, let $L_{t,N}$ be the time-inhomogeneous Markov generator on $\R^{2N}$
given by
\[L_{t,N}  = x\cdot \na_v -(x+\gamma v) \cdot \na_v - \sum_{i=1}^N F_{f}(t,x_i) \cdot \na_{v_i} + \gamma \Delta_v  \,,  \]
which corresponds to $N$ independent copies of the solution on $\R^{2}$ of
\[
 \left\{
\begin{array}{rcl}
\dd X_t &=& V_t\dd t\\
\dd V_t &=  & -\po X_t + \gamma V_t + \lambda F_f(t,X_t)\pf \dd t + \sqrt{2} \dd B_t\,. 
\end{array}
\right.
\]
Since $f$ solves \eqref{eq:VFP}, $(X_t,V_t) \sim f(t,\cdot)$ if $(X_0,V_0) \sim f(0,\cdot)$. In particular, for any nice function $g$ on $\R^{2N}$,
\[ \partial_t \int_{\R^{2N}} g f^{\otimes N}(t,\cdot) = \int_{\R^{2N}} L_{t,N} g f^{\otimes N}(t,\cdot)\,.\]
Similarly, let $L_N$ be the generator associated with the interacting particle system \eqref{eq:mean_field_particle_2}, i.e.
\[L_N = x\cdot \na_v -(x+\gamma v) \cdot \na_v - \frac\lambda N \sum_{i,j=1}^N \partial_x K(x_i-x_j) \cdot \na_{v_i} + \gamma \Delta_v\,. \]
Since $\mu_N$ is invariant for $L_N$, $\int L_N g \mu_N = 0$ for all suitable $g$. Then, using that mass is preserved along \eqref{eq:VFP} so that $\int \partial_t f = 0$,
\begin{eqnarray*}
\partial_t \int_{\R^{2N}}  f^{\otimes N} \ln \frac{f^{\otimes N}}{\mu_N} & =&   \int_{\R^{2N}} f^{\otimes N} L_{f,N} \po   \ln \frac{f^{\otimes N}}{\mu_N} \pf  + \partial_t \po f^{\otimes N}\pf \\
& = & \int_{\R^{2N}}  f^{\otimes N} \po L_{f,N} - L_N\pf \po   \ln \frac{f^{\otimes N}}{\mu_N} \pf\\ 
& & \ + \int_{\R^{2N}} \co \frac{f^{\otimes N}}{\mu_N} L_{N} \po   \ln \frac{f^{\otimes N}}{\mu_N} \pf - L_N \po \frac{f^{\otimes N}}{\mu_N} \ln \frac{f^{\otimes N}}{\mu_N}\pf\cf \mu_N  \,.
\end{eqnarray*}
Denoting by $(*)$ and $(**)$ respectively the integrals in the  last two lines of this equation, we recognize for the latter the classical entropy dissipation for \eqref{eq:mean_field_particle_2}, for which standard computations give
\[(**) \leqslant - \int_{\R^{2N}} \left| \na_v \ln \frac{f_t^{\otimes N}}{\mu_N}\right |^2 f_t^{\otimes N}\,.\]
For the former, denoting $b=(b_1,\dots,b_N)$ with
\[b_i(x) =  \frac1  {N-1} \sum_{j\in\cco 1,N\ccf \setminus\{i\}} \po \partial_x K(x_i-x_j) - \partial_x K \star f(x_i)\pf \,,  \]
we get
\[(*) = \lambda \int_{\R^{2N}}  f^{\otimes N} b \cdot \na_v    \ln \frac{f^{\otimes N}}{\mu_N}\,. \]
Using the Cauchy-Schwarz inequality,
\[\partial_t \int_{\R^{2N}}  f^{\otimes N} \ln \frac{f^{\otimes N}}{\mu_N} \leqslant \frac{\lambda^2} 4 \int_{\R^{2N}}  |b|^2 f^{\otimes N} \,.   \] 
 For $i\in\cco 1,N\ccf$, using the independence of particles and that $x_j\mapsto \partial_x K(x_i-x_j) - \partial_x K \star f(x_i)$ is centered under the law $f$,
 \begin{eqnarray*}
 \int_{\R^{2N}}  |b_i|^2 f^{\otimes N} &=& \frac1{(N-1)^2} \sum_{j\in\cco 1,N\ccf\setminus\{i\}} \int_{\R^{2N}} \po \partial_x K(x_i-x_j) - \partial_x K \star f(x_i)\pf^2 \\
 & \leqslant & \frac2{N-1} \|\partial_x^2 K\|_\infty \int_{\R^2} x^2 f(t,x,v)\dd x\dd v\,.
 \end{eqnarray*}
 It is easily seen that $\partial_t \int_{\R^2} |z|^2 f(t,z)\dd z \leqslant C \int_{\R^2} |z|^2 f(t,z)\dd z$ for some $C$. As a conclusion, for $t\geqslant s\geqslant 0$,
 \[\int_{\R^{2N}}  f_t^{\otimes N} \ln \frac{f_t^{\otimes N}}{\mu_N}  \leqslant \int_{\R^{2N}}  f_s^{\otimes N} \ln \frac{f_s^{\otimes N}}{\mu_N} + C_t\]
 for some $C_t\geqslant 0$. Dividing by $N$ and taking the $\liminf$ concludes.
 \end{proof}

 \begin{proof}[Proof of Proposition~\ref{prop:Fisher}]
Throughout the proof we keep computations formal, assuming sufficient regularity and integrability. See the approximation argument of \cite{Songbo} for a rigorous justification of similar computations.
 
We decompose 
\[\partial_t \int_{\R^{2}} \left|A \na \ln \frac{f_t}{\hat f_t} \right|^2 f_t = \partial_s \int_{\R^{2}} \left|A\na \ln \frac{f_{t+s}}{\hat f_t} \right|^2 f_{t+s} + \partial_s \int_{\R^{2}} \left|A\na \ln \frac{f_t}{\hat f_{t+s}} \right|^2 f_t = (\star)+(\star\star)\,. \]
Since $\hat f_t$ is the invariant measure of $L_{f_t}$, the first part leads to classical computations (e.g. \cite[Lemma 8]{Gamma}) which yield
\[(\star) \leqslant  2 \int_{\R^{2}}  \na \ln \frac{f_t}{\hat f_t} \cdot A^T A J   \na  \ln \frac{f_t}{\hat f_t}   f_t \,,\]
where
\[J(x) = \begin{pmatrix}
0 &  1 + \lambda \partial_x F_{f_t}(x) \\- 1 & -\gamma 
\end{pmatrix} =: J_0 + \lambda \partial_x F_{f_t}(x) \begin{pmatrix}
0 & 1 \\ 0 & 0
\end{pmatrix} \]
is the Jacobian matrix of the drift of 
\[L_{f_t}^* =- v\cdot \na_x + (x+\partial_x K\star f(x)-\gamma v)\cdot\na_v + \gamma \Delta_v\,,\]
the dual of $L_{f_t}$ in $L^2 (\hat f_t)$. We bound
\[(\star)  \leqslant  2\int_{\R^{2}} \co   \na \ln \frac{f_t}{\hat f_t} \cdot A^T A J_0   \na  \ln \frac{f_t}{\hat f_t}    + \lambda' \left| A  \na \ln \frac{f_t}{\hat f_t} \right| \left|  \na_v \ln \frac{f_t}{\hat f_t}\right|\cf f_t \]
with $\lambda' = \lambda\|\partial_x^2 K\|_\infty \sqrt{A_{11}^2+A_{21}^2}$. 

 We now deal with $(\star\star)$. Note that this question is addressed in the proof of Theorem~2.1 of \cite{Songbo}. We give slightly different  computations. First, from the explicit expression \eqref{eq:fhat} of $\hat f_t$,  $\partial_t \na_v \ln \hat f_t(x,v)=0$ and
\begin{eqnarray*}
\partial_t \na_x \ln \hat f_t(x,v)& =& - \lambda\int_{\R^2} \partial_x K(x-y) \partial_t f_t (y,w) \dd y \dd w \\
& = &  -\lambda \int_{\R^2} L_{f_t} \po (y,w) \mapsto \partial_x K(x-y)\pf  f_t (y,w) \dd y \dd w\\
& = &  \lambda \int_{\R^2} w \partial_x^2 K(x-y)  f_t (y,w) \dd y \dd w\\
& = &  \lambda \int_{\R^2} w \partial_x^2 K(x-y)  \frac{f_t (y,w)}{\hat f_t(y,w)} \hat f_t(y,w) \dd y \dd w\\
& = &  - \lambda \int_{\R^2}  \partial_x^2 K(x-y)  \frac{f_t (y,w)}{\hat f_t(y,w)} \na_w \hat f_t(y,w) \dd y \dd w\\
& = &\lambda   \int_{\R^2}  \partial_x^2 K(x-y) \na_w  \frac{f_t (y,w)}{\hat f_t(y,w)}  \hat f_t(y,w) \dd y \dd w \\
& = &  \lambda \int_{\R^2}  \partial_x^2 K(x-y) \na_w  \ln \po  \frac{f_t (y,w)}{\hat f_t(y,w)} \pf  f_t(y,w) \dd y \dd w \,.
\end{eqnarray*}
As a consequence, for all $(x,v)\in\R^2$,
\[|\partial_t A \na \ln \hat f_t(x,v)|^2 \leqslant (\lambda')^2    \int_{\R^2}  \left|  \na_w  \ln  \frac{f_t}{\hat f_t} \right|^2   f_t \,,\]
and thus 
\begin{eqnarray*}
(\star\star) & \leqslant &  2 \lambda' \co \int_{\R^2}  \left|A  \na  \ln  \frac{f_t}{\hat f_t} \right|^2   f_t   \int_{\R^2}  \left|  \na_w  \ln  \frac{f_t}{\hat f_t} \right|^2   f_t \cf^{1/2} \\
& \leqslant & \lambda '   \int_{\R^{2}}  \na \ln \frac{f_t}{\hat f_t} \cdot \co  \theta A^T A  + \theta^{-1} \begin{pmatrix}
0 & 0 \\ 0 & 1
\end{pmatrix} \cf    \na  \ln \frac{f_t}{\hat f_t}   f_t
\end{eqnarray*}
for any $\theta>0$. We end up with
\[\partial_t \int_{\R^{2}} \left|A \na \ln \frac{f_t}{\hat f_t} \right|^2 f_t \leqslant  2 \int_{\R^{2}}  \na \ln \frac{f_t}{\hat f_t} \cdot \co  A^T A \po J_0 +   \lambda '   \theta I\pf  +   \lambda '  \theta^{-1} \begin{pmatrix}
0 & 0 \\ 0 & 1
\end{pmatrix} \cf    \na  \ln \frac{f_t}{\hat f_t}   f_t \,. \]

As discussed in \cite[Section 2.1]{Contraction} the matrix $A^T A$ that we should take here to get a contraction is related to the modified Euclidean distance used along the coupling in the  proof of Proposition~\ref{prop:cv}. As a consequence, considering $a=\min(\gamma,\gamma^{-1})/2$ and $b=1+a^2 - a\gamma$ as in the latter, we take $A$ to be the symmetric square root of $A A^T = M^{-1}$ where
\[M = \begin{pmatrix}
1 &- a \\ -a & b + a^2 
\end{pmatrix} \,.\]
Notice that the norm used in the proof of Proposition~\ref{prop:cv} is the square-root of $(x,v) \mapsto (x,-v)^T M (x,-v)$ (The velocity flip is due to the fact we work here with $L_{f_t}^*$ instead of $L_{f_t}$, i.e. we work with the relative density $f_t/\hat f_t$ instead of the Markov semi-group associated to $L_{f_t}$). Indeed, revisiting the computations of the proof of Proposition~\ref{prop:cv}, we see that that, for any   $u\in\R^{2}$, writing $v=M^{-1} u$,
\begin{eqnarray*}
u^T A A^T J_0 u  & = & v^T J_0 M v \\
& = & v^T \begin{pmatrix}
-a  &  b+ a^2 \\ -1+\gamma a & a - \gamma(b+a^2)
\end{pmatrix} v\\
& = & v^T \co -a \begin{pmatrix}
1 & -a \\ -a & a^2 
\end{pmatrix} + b(a-\gamma) \begin{pmatrix}
0 & 0 \\ 0 & 1 
\end{pmatrix}  \cf v\\
& = & -a |v_1 - a v_2|^2 + b(a-\gamma)|v_2|^2\,,
\end{eqnarray*}
to be compared to \eqref{eq:mult}. Following the computations of the proof of Proposition~\ref{prop:cv}, under \eqref{eq:lambda_condition}, we get that, for a suitable choice of $\theta$,
\[u^T  \co  A^T A \po J_0 + \lambda '   \theta I\pf  +   \lambda '  \theta^{-1} \begin{pmatrix}
0 & 0 \\ 0 & 1
\end{pmatrix} \cf u \leqslant  - \frac{a}{8} v^T M v = - \frac{a}{8}|A u|^2\,.   \]

In other words, with this choice for $A$,
\[\partial_t \int_{\R^{2}} \left|A \na \ln \frac{f_t}{\hat f_t} \right|^2 f_t \leqslant - \frac{a}{4} \int_{\R^{2}} \left|A \na \ln \frac{f_t}{\hat f_t} \right|^2 f_t \,,\]
which concludes. 
\end{proof}

\subsection*{Acknowledgments}

This work has been supported by the projects CONVIVIALITY (ANR-23-CE40-0003) and SWIDIMS (ANR-20-CE40-0022) of the French National Research Agency.


\end{document}